\documentclass[12pt,twoside]{article}
\usepackage{amsmath,amssymb,mathrsfs,graphicx,bbm,amsfonts,amsthm,mathtools}
\mathtoolsset{showonlyrefs}
\usepackage[utf8x]{inputenc}
\usepackage{nicefrac}
\usepackage{color}
\usepackage[normalem]{ulem}
\usepackage[english]{babel}

\usepackage{graphicx,subcaption}
\usepackage{url}
\usepackage{blindtext}

\usepackage[colorlinks=true, citecolor = green, linkcolor=blue, breaklinks, pdfauthor ={"Elena Magnanini, Giacomo Passuello"}]{hyperref}
\usepackage{todonotes}

\usepackage{comment}

\usepackage{graphicx}
\usepackage[title]{appendix}%

\newtheorem{theorem}{Theorem}[section]
\newtheorem{lemma}[theorem]{Lemma}
\newtheorem{proposition}[theorem]{Proposition}
\newtheorem{corollary}[theorem]{Corollary}

\newtheorem{remark}[theorem]{Remark}

\DeclareMathAlphabet{\mathbfit}{OML}{cmm}{b}{it}

\makeatletter
\@addtoreset{equation}{section}
\makeatother

\usepackage{titling}
    \pretitle{\centering\LARGE\scshape}
        \posttitle{\vskip 0.75cm}

    \predate{\vskip 0.75 cm \centering\large}
        \postdate{\par}

\newcommand{\ER}{Erd\H{o}s-R\'{e}nyi }

\newcommand{\cadlag}{\ifmmode\mathcal D\else c\`adl\`ag \fi}

\newcommand{\ba}{\begin{array}}
	\newcommand{\ea}{\end{array}}
\newcommand{\bea}{\begin{eqnarray}}
	\newcommand{\eea}{\end{eqnarray}}
\newcommand{\be}{\begin{equation}}
	\newcommand{\ee}{\end{equation}}

\newcommand{\N}{\mathbb{Z}_{+}}

\newcommand{\trn}{\lfloor\nicefrac{T_n}{n}\rfloor}
\newcommand{\tr}{\nicefrac{T_n}{n}}

\def \C {{\mathbb C}}
\def \N {{\mathbb N}}

\def \E {{\mathbb E}}

\def \cA {\mathcal{A}}

\def \cE {\mathcal{E}}

\def \cG {\mathcal{G}}

\def \cP {\mathcal{P}}

\def \cT {\mathcal{T}}

\def \a {{\alpha}}
\def \b {{\beta}}

\def \z {{\z}}

\def \z {{\zeta}}

\def \1{\mathbbm{1}} 

\author{
Elena Magnanini \thanks{WIAS Berlin. \\ magnanini@wias-berlin.de} 
\and
Giacomo Passuello \thanks{Dipartimento di Matematica \textquotedblleft Tullio Levi-Civita",
Universit\`a di Padova.\\
giacomo.passuello@phd.unipd.it}
}
\title{A standard clt for  triangles in a class of ergs}
\date{\today}

\usepackage[top=2.5cm, bottom=2.5cm,left=2.5cm,right=2.5cm]{geometry}

\begin{document}
\maketitle

\begin{abstract}
We prove a standard Central Limit Theorem for the (normalized) number of triangles in a class of Exponential Random Graphs derived from a slight modification of the edge-triangle model. 
Our main theorem covers the whole analyticity region of the free energy, and is based on a polynomial representation of the partition function.   
 
\end{abstract}

\section{Introduction}
Exponential random graphs (ERGs) are a widely studied class of models that aim to incorporate typical tendencies, such as clustering, commonly observed in real networks.
As a generalization of the \ER model, ERGs allow for dependencies between edges. This is done by using the statistical mechanics approach of introducing an \emph{Hamiltonian} to bias the probability measure over the space of graphs, 
enhancing or penalizing the density of specific subgraph counts.  
We refer the reader to  \cite{Ch} for a comprehensive overview. 
From a statistical mechanics perspective, ERGs can be interpreted as finite spin systems, where each potential edge corresponds to a spin variable taking values in $\{0,1\}$. The absence of symmetry, typical in classical spin models, adds an additional layer of complexity and makes these graphs particularly interesting in many respects.
A broad line of literature, spanning both classical and more recent works, has developed around limit theorems and concentration inequalities for sums of dependent variables, particularly in the contexts of spin systems and random graph models (see, e.g., \cite{EN,EL, CDey}).
In the context of ERGs, several such results have been established, though the majority of the literature has concentrated on edge density, with significantly less known about the behavior of higher-order subgraph counts such as triangles.
A Central Limit Theorem (CLT) for the edge density was first established in \cite{BCM} for a
specific class of ERGs known as the edge-triangle model, where the Hamiltonian depends exclusively on the edge and triangle densities. This result was later generalized in \cite{fang2024normal} to
a broader class of ERGs, using Stein’s method (see \cite{Stein}). However, the analysis is restricted to a specific parameter regime known as Dobrushin's uniqueness region (see, e.g., \cite[Eq. (2)]{fang2024normal}), which the result in \cite{BCM} can include and extend beyond.  
Finally, we mention the work of \cite{MX}, where a CLT for the edge density is established for the two-star model, a class of ERGs in which edge dependencies arise from the presence of two-stars (i.e. subgraphs consisting of three vertices with two edges sharing a common vertex).
Beyond the edge statistics, we point out the very recent work of \cite{fang2025conditional}, which studies the asymptotic distribution of the number of two-stars in a model of ERG where the number of edges is conditioned to satisfy some constraint.\\
As mentioned, quite less is known about the fluctuations of the triangle density, which is the focus of this paper.
We are aware of only a single work from \cite{SaSi}. Even there, the results are confined to the Dobrushin's region as the aforementioned edge-based results, based on Stein's method. 

\subsection{Our contribution}
In this paper, 
we prove a standard CLT for the normalized number of triangles in a class of ERGs obtained through a slight modification of the edge-triangle model. 
Our main theorem covers the whole analyticity region of the limiting free energy, which goes  beyond the aforementioned Dobrushin’s uniqueness region.
The technique of Thm. \ref{thm_CLT} can be easily extended to other subgraph counts but also, in principle, to more general families of ERGs, provided that the phase diagram of the free energy is known (as it happens for the 3-parameter model, see Thm. \ref{thm_CLT3p}).
 
\section{The model}
Let $\cG_n$ be the set of all simple graphs on $n$ labeled vertices
that are identified with the elements of the set $[n]=\{1,2,3,\ldots, n\}$.   
In ERGs, the probability distribution on $\cG_n$ is defined via a function, called Hamiltonian, contained in an exponential term, which collects the \textit{homomorphism densities} of the subgraphs of a graph.  
A homomorphism from a fixed simple graph $H$ to a graph $G\in \mathcal{G}_n$ is a map from the vertex set $V(H)$ to the vertex set $V(G)$ that preserves adjacency, i.e.  it maps edges of $H$ to edges of $G$. Denoting the number of such maps by $|\mathrm{hom}(H, G)|$, we define the homomorphism density of $H$ into $G$ as
\begin{equation}\label{def_graph_hom_density}
t(H,G) := \frac{|\mathrm{hom}(H,G)|}{|V(G)|^{|V(H)|}},
\end{equation}
which represents the probability that a uniformly chosen mapping from $V(H)$ to $V(G)$ is edge-preserving.
Given finite simple graphs  $H_1, \dots, H_k$ (e.g., edges, stars, triangles, cycles, etc.) and a parameter vector $ \boldsymbol{\beta} = (\beta_1, \dots, \beta_k) \in \mathbb{R}^k $, the exponential random graph model (ERGM) assigns to each graph $G \in \mathcal{G}_n$ the Gibbs probability density
\begin{equation}\label{eq:prob-exprg}
\mu_{n; \boldsymbol{\beta}} (G) = \frac{\exp \left(\mathcal{H}_{n,\boldsymbol{\beta}}(G)\right)}{Z_{n;\boldsymbol{\beta}}},
\end{equation}
with Hamiltonian
\begin{equation}\label{Hamiltonian}
\mathcal{H}_{n;\boldsymbol{\beta}}(G) = n^2 \sum_{j=1}^{k} \beta_j t(H_j, G),
\end{equation}
where \begin{equation}\label{eq:partit-expon}
Z_{n;\boldsymbol{\b}}=\sum_{G \in \cG_{n}} \exp \left(H_{n;\boldsymbol{\b}}(G)\right) 
\end{equation}
is the normalizing \emph{partition function}. 
We specialize to the case where $H_1$ is an edge, $H_2$ a triangle and $\beta_3,\beta_4, \ldots, \beta_k=0$. Thus, the homeomorphism densities turn out to be $t(H_{1},G)=\frac{2E_n}{n^2}$ and $t(H_{2},G)=\frac{6T_n}{n^3}$, where $E_n$ and $T_n$ denote, respectively, the total number of edges and triangles in $G$. This specific setting gives rise to the so-called edge-triangle model.
It is more convenient to switch to a representation of the model in terms of adjacency matrices.
We denote by $\mathcal{E}_n$ the edge set of the complete graph
	on $n$ vertices, with elements labeled from 1 to $\binom{n}{2}$ and we set $\mathcal{A}_n := \{0,1\}^{\mathcal{E}_n}$, which is in one-to-one correspondence with the graphs in $\cG_n$. As a consequence, to each graph $G\in\cG_n$ we can associate an element ${x}=(x_{i})_{i \in \cE_n}\in\cA_n$ where $x_{i}=1$ if the edge $i$ is present in $G$, and $x_{i}=0$ otherwise. Now, we see the Hamiltonian in \eqref{Hamiltonian} as a function on $\cA_n$, and we obtain the equivalent formulation:
 	%
	\begin{equation}\label{Hamilt_ERG}
		\mathcal{H}_{n;\alpha,h}(x) = \frac{\alpha}{n} \sum_{\{i,j,k\} \in \mathcal{T}_n} x_i x_j x_k + h \sum_{i \in \mathcal{E}_n} x_i,
	\end{equation}
where $\mathcal{T}_n:=\{\{i,j,k\} \subset \mathcal{E}_n: \{i,j,k\} \text{ is a triangle}\}$, and for convenience we set $h:=2\b_1$ and $\a:=6\b_2$.
Inside the domain $\mathcal{D}_{\alpha,h}^{rs}:=\{ (\alpha, h) : \alpha >-2, h\in\mathbb{R} \}$, which is called \emph{replica symmetric} regime, the limiting free energy  $f_{\alpha,h}:= \lim_{n\to \infty}\frac{1}{n^2}\ln({Z}_{n;\alpha,h})$ associated with the Hamiltonian \eqref{Hamilt_ERG}, can be obtained as the solution of a scalar problem.
\subsection{Free energy} \label{known-res}
It is well-known (see \cite[Thms. 4.1-6.1]{CD}), that  if  $\alpha,h \in \mathcal{D}_{\alpha,h}^{rs}$, then 
	\begin{equation}\label{free_energy}
		f_{\a,h} = \,  \sup_{0 \leq u \leq 1} \left(\frac{\alpha}{6}u^3 + \frac{h}{2}u - \frac{1}{2}I(u) \right) \, = \, \frac{\a}{6}{u^{*}}^3 +\frac{h}{2}u^{*} -\frac{1}{2}I(u^{*}),
	\end{equation}
	where $I(u):=u\ln u + (1-u)\ln(1-u)$ and $u^{*}$ is a maximizer that solves the fixed-point equation
	\begin{equation}\label{FixPointEq}
		\frac{e^{\alpha\,u^{2} +h}}{1+e^{\alpha\,u^{2} +h}}=u\,.
	\end{equation}
For the sake of readability, we will sometimes omit the dependence of \( u^* \) on \( \alpha \) and \( h \), and write simply \( u^* \); when we wish to emphasize this dependence, we will write \( u^*_{\alpha,h} \).

Our results focus on the following modification of \eqref{Hamilt_ERG}, where we take into account only the integer part of the normalized number of triangles:

	\begin{equation}\label{H_norm}
		\hat{\mathcal{H}}_{n;\alpha,h}(x):= \alpha \left\lfloor\frac{\sum_{\{i,j,k\} \in \mathcal{T}_n} x_i x_j x_k}{n}\right\rfloor + h \sum_{i \in \mathcal{E}_n} x_i.
	\end{equation}
We denote by $\hat{\mu}_{n;\alpha,h}$ the associated Gibbs probability density and by $\hat{\mathbb{P}}_{n;\alpha,h}$, $\hat{\mathbb{E}}_{n;\alpha,h}$  the related measure, with normalizing partition function $\hat{Z}_{n;\alpha,h}$, and expectation. Finally, we indicate by 
\begin{equation}\label{hat_f}
\hat{f}_{n;\alpha,h}:= \frac{1}{n^2}\ln\hat{Z}_{n;\alpha,h} \quad \text{ and } \quad \hat{f}_{\alpha,h}:=\lim_{n \to +\infty}\hat{f}_{n;\alpha,h}
\end{equation}
the finite-size and the limiting free energy, respectively.\\
Importantly, $\hat{f}_{\alpha,h}=f_{\alpha,h}$. This immediately follows from the decomposition $\frac{\sum_{\{i,j,k\} \in \mathcal{T}_n} x_i x_j x_k}{n}= \left\lfloor\frac{\sum_{\{i,j,k\} \in \mathcal{T}_n} x_i x_j x_k}{n}\right\rfloor +\left\{\frac{\sum_{\{i,j,k\} \in \mathcal{T}_n} x_i x_j x_k}{n}\right\}$, where $\{\cdot\}\in [0,1]$ denotes the fractional part.

\subsection{Phase diagram of the free energy}
Inside the domain $\mathcal{D}_{\alpha,h}^{rs}$, the limiting free energy is analytic except for a critical curve, which we denote by $\mathcal{M}^{rs}$, that starts at the critical point $(\alpha_c,h_c) := \left(\frac{27}{8},\ln 2 -\frac{3}{2}\right)$  and can be written as $h=q(\alpha)$ for a (non-explicit) continuous and strictly decreasing function $q$: 	
$\mathcal{M}^{rs} := \left\{(\alpha,h) \in (\alpha_c,+\infty) \times (-\infty,h_c): h = q(\alpha)\right\}$ (see \cite[Prop. 3.8]{RY}).
 It is very important for our result to know how to characterize the analyiticity region, which can be expressed, in this notation, as 
 $\mathcal{U}_{\alpha,h}^{rs}\setminus \{(\alpha_c,h_c)\}$, where
		$\mathcal{U}_{\alpha,h}^{rs} := \mathcal{D}_{\alpha,h}^{rs} \setminus \mathcal{M}^{rs}$.
The explicit description of the analyticity region is also available for a 3-parameter ERGM, under some restrictions. We report the result in Thm. \ref{3param_r} below, as we are going to extend our main theorem to this setting.
Assume, 
more precisely, that \( H_1 \) is a single edge, \( H_2 \) has \( p \) edges, and \( H_3 \) has \( q \) edges, with \( 2 \leq p \leq q \leq 5p - 1 \). Let $f_{\beta_1, \beta_2,\beta_3}$ the limiting free energy arising from the Hamiltonian \eqref{Hamiltonian} by setting
$\beta_{k}=0$  for all $k\geq 4$. Such function, inside the domain  \(\mathcal{D}^{rs}_{\beta_1,\beta_2,\beta_3}:= \{ (\beta_1, \beta_2, \beta_3) : \beta_2 \geq 0, \beta_3 \geq 0, \beta_1\in\mathbb{R} \} \) (again, by \cite[Thm. 4.1]{CD}) exists and equals \footnote{We stress that, from \cite[Thm. 6.1]{CD}, $\mathcal{D}^{rs}_{\beta_1,\beta_2,\beta_3}$ is actually a subset of the region where the free energy is known.  However, Thm. \ref{3param_r} only applies to this restriction.}
	\begin{equation}\label{free_energybis}
		f_{\beta_1,\beta_2,\beta_3} = \,  \sup_{0 \leq u \leq 1} \left(\beta_3 u^{q} + \beta_2 u^{p}+ \beta_1 u - \frac{1}{2}I(u) \right) \, = \, \beta_3 {u^{*}}^q +\beta_2{u^{*}}^p+ \beta_1 u^{*} -\frac{1}{2}I(u^{*}),
	\end{equation}
where $u^*$ solves 
\begin{equation}\label{ustar_ext}
u = \frac{e^{2\beta_3 q u^{q-1} + 2\beta_2 p u^{p-1} + 2\beta_1}}{1 + e^{2\beta_3 q u^{q-1} + 2\beta_2 p u^{p-1} + 2\beta_1}}.
\end{equation}
The phase diagram in this setting is also known.
\begin{theorem}[\cite{MY}, Thm.~1] \label{3param_r}
The free energy $f_{\beta_1,\beta_2,\beta_3}$ is analytic in $\mathcal{D}^{rs}_{\beta_1,\beta_2,\beta_3}$ except for a certain continuous surface \( S \) which includes three bounding curves \( C_1 \), \( C_2 \), and \( C_3 \), and that can be characterized as follows:
\begin{itemize}
\item the surface \( S \) approaches the plane \( \beta_1 + \beta_2 + \beta_3 = 0 \) as \( \beta_1 \to -\infty \), \( \beta_2 \to \infty \), and \( \beta_3 \to \infty \);

\item the curve \( C_1 \) is the intersection of \( S \) with the \( (\beta_1, \beta_2) \)-plane, i.e., \( \{ (\beta_1, \beta_2, \beta_3) : \beta_3 = 0 \} \);

\item the curve \( C_2 \) is the intersection of \( S \) with the \( (\beta_1, \beta_3) \)-plane, i.e., \( \{ (\beta_1, \beta_2, \beta_3) : \beta_2 = 0 \} \);

\item the curve \( C_3 \) is a critical curve, and is given parametrically by
\begin{align*}
\beta_1(u) &= \frac{1}{2} \ln \frac{u}{1 - u} - \frac{1}{2(p - 1)(1 - u)} + \frac{pu - (p - 1)}{2(p - 1)(q - 1)(1 - u)^2}\\
\beta_2(u) &= \frac{qu - (q - 1)}{2p(p - 1)(p - q) u^{p - 1} (1 - u)^2}\\
\beta_3(u) &= \frac{pu - (p - 1)}{2q(q - 1)(q - p) u^{q - 1} (1 - u)^2},
\end{align*}
where we take \( \frac{p - 1}{p} \leq u \leq \frac{q - 1}{q} \) to meet the non-negativity constraints on \( \beta_2 \) and \( \beta_3 \).
\end{itemize}

\end{theorem}

\section{Main result}
With a slight abuse of notation, in the following we denote by $T_n$ the random number of triangles of a graph sampled according to $\hat{\mathbb{P}}_{n;\alpha,h}$.
\begin{theorem}[CLT for $T_n$ w.r.t. $\hat{\mathbb{P}}_{n;\alpha,h}$]\label{thm_CLT}
For all $(\alpha,h)\in\mathcal{U}_{\alpha,h}^{rs}\setminus\{(\alpha_c,h_c)\}$
$$	
 \sqrt{6} \, \frac{\nicefrac{T_n}{n} -\hat{\mathbb{E}}_{n;\alpha,h}(\nicefrac{T_n}{n})}{n} \xrightarrow{\;\;\mathrm{d}\;\;}{} \mathcal{N}(0,v(\alpha,h)) \quad \text{ w.r.t. } \hat{\mathbb{P}}_{n;\a,h},\text{ as } n \to +\infty,$$
where $v(\alpha,h):= 3u^{*2}_{\alpha,h}\partial_{\alpha} u^{*}_{\alpha,h}$ and $\mathcal{N}(0,v(\alpha,h))$ is a centered Gaussian distribution.
\end{theorem}
The theorem immediately extends to the following setting.
Consider the 3-parameter Hamiltonian obtained from \eqref{Hamiltonian} by taking \( H_1 \) a single edge, \( H_2 \) a triangle, and \( H_3 \) a simple subgraph with $q \in[3,14]$ edges, and setting $\beta_k=0$ for all $k\geq 4$. Similarly to what we did for the edge-triangle case, we denote by 
\begin{equation}\label{Ham_3p}
\hat{\mathcal{H}}_{n;\beta_1,\beta_2,\beta_3}(x)= \beta_3n^2 t(H_3,x) + \beta_2 \lfloor n^2t(H_2,x)\rfloor +\beta_1 n^2t(H_1,x),
\end{equation}
and by  $\hat{\mathbb{P}}_{n;\beta_1,\beta_2,\beta_3}$
 the associated Gibbs measure. Then, the following result also holds.

\begin{theorem}[CLT for $T_n$ w.r.t. $\hat{\mathbb{P}}_{n;\beta_1,\beta_2,\beta_3}$]\label{thm_CLT3p}
 For all $(\beta_1,\beta_2,\beta_3)\in\mathcal{D}^{rs}_{\beta_1,\beta_2,\beta_3}\setminus S$

$$	
\sqrt{6} \, \frac{T_n/n - \hat{\mathbb{E}}_{n;\beta_1,\beta_2,\beta_3}(\nicefrac{T_n}{n})}{n} \xrightarrow{\;\;\mathrm{d}\;\;}{} \mathcal{N}(0,v(\beta_1,\beta_2,\beta_3)) \quad \text{ w.r.t. } \hat{\mathbb{P}}_{n;\beta_1,\beta_2,\beta_3},\text{ as } n \to +\infty,$$
where $v(\beta_1,\beta_2,\beta_3):=18u^{*2}_{\beta_1,\beta_2,\beta_3}\partial_{\beta_2} u^{*}_{\beta_1,\beta_2,\beta_3}$  and $u^*$ solves \eqref{ustar_ext}.
\end{theorem}

Note that, unlike in the definition of \( \alpha, h \) (see below \eqref{Hamilt_ERG}), this generalization retains a constant factor within the homomorphism density. This explains why the two variances in Thm. \ref{thm_CLT3p} and Thm. \ref{thm_CLT} differ by a factor of $6$. 

\begin{remark}
Some remarks are in order.
\begin{enumerate}
\item Based on the result proved in \cite[Thm. 3.6]{MP}, which considers an approximation of triangle counts in a mean-field version of the edge-triangle model, we conjecture that
$v(\alpha,h)= \frac{3{{u_{\alpha,h}^{*4}}}}{4c_0}$, where $c_0\equiv c_0(\alpha,h):=\frac{1-2\alpha {{u_{\alpha,h}^{*2}}}(1-u_{\alpha,h}^{*})}{4\alpha {u_{\alpha,h}^{*}}(1-{u_{\alpha,h}^{*}})}$.
\item The choice of including triangles in the statistics of \eqref{H_norm} and \eqref{Ham_3p} is crucial, as it allows us to connect the expectation of  $\lfloor\nicefrac{T_n}{n}\rfloor$ and hence the scaled cumulant generating function defined in \eqref{cgf} below, to the derivative of the finite-size free energy. This will be a key step in the proof of Thm. \ref{thm_CLT}.
\item  The integer part of the normalized triangle count plays a crucial role when we represent the partition function as a polynomial (see Subsec. \ref{part_Fct_repr}).
The other subgraph counts collected in the Hamiltonian (see e.g. \eqref{Ham_3p}), can be taken without such integer value, as
they contribute only to the coefficients of the polynomial, and do not affect the validity of the representation.
\end{enumerate}
\end{remark}

\section{Proofs}
This section is dedicated to the proof of Thm.~\ref{thm_CLT}; the proof of Theorem~\ref{thm_CLT3p} is omitted, as it follows exactly the same argument.
To describe the fluctuations of $\trn$ around its mean value,  in view of the decomposition $\nicefrac{T_n}{n}= \lfloor\nicefrac{T_n}{n} \rfloor +\{\nicefrac{T_n}{n}\}$ combined with Slutsky's theorem (see \cite[Thm. 13.8]{Klenke}), it is enough to study the asymptotic behavior of the moment generating function of  $W_n:=\sqrt{6} \, \frac{\lfloor\nicefrac{ T_n}{n}\rfloor -\hat{\mathbb{E}}_{n;\alpha,h}(\lfloor\nicefrac{T_n}{n}\rfloor)}{n}$. Specifically, we are going to 
relate such generating function to the second order derivative of the cumulant generating function of $\lfloor\nicefrac{T_n}{n}\rfloor$, which is defined as 
\begin{equation}\label{cgf}
c_{n}(t):= 6n^{-2}\ln\hat{\mathbb{E}}_{n;\a,h}[\exp(t \trn)], \qquad t\in\mathbb{R}.
\end{equation}
\begin{remark}
Note that, by a direct calculation, we get
	\begin{equation}\label{link_c_avg&var_dens}
		c'_n(t)=\frac{6\hat{\mathbb{E}}_{n;\a+t,h}\left(\trn\right)}{n^{2}} \quad \text{ and } \quad c''_n(t)=\frac{6\text{Var}_{n;\a+t,h}(\trn)}{n^{2}}.
	\end{equation}
\end{remark}

The limit of the sequence $(c''_{n}(t))_{n\geq 1}$ for $t=t_{n}=o(1)$ will give the variance of the limiting Gaussian. The existence of this limit follows from the Yang–Lee theorem (see Thm. \ref{teo_LY}). To apply it, we first need a suitable representation of the partition function, which we provide as a first step.  Following this, we establish some auxiliary results that will be used in the proof, which is deferred to the end of the section.

\subsection{Representation of the partition function.}\label{part_Fct_repr}
We start from the partition function obtained by plugging \eqref{Hamilt_ERG} into \eqref{eq:partit-expon}, and then we incorporate the integer part. First, we have:
	\begin{align}
		Z_{n;\a,h}&= \sum_{x\in\mathcal{A}_{n}}e^{\frac{\alpha}{n}
			\sum_{\{i,j,k\}\in \cT_n}x_{i}x_{j}x_{k} + h \sum_{i\in\cE_n}x_{i}}\,.
	\end{align}
	
	Notice that there is a bijection between $\cA_n$
	and the 
	power set $\cP(\cE_n)$,
	that maps an element $x\in \cA_n$ to the set
	$S=\{i\in\cE_n\,:\, x_{i}=1\}$.
	We can then decompose $\cA_n$ in disjoint subsets as
	$$\cA_n
	= \bigcup_{m=0}^{\binom{n}{3}}  \: \bigcup_{\ell=0}^{\binom{n}{2}} \: 
	\bigcup_{\substack{S\subseteq \cE_n: |S|=\ell, \\ |\{\{i,j,k\}\subset S\,:\,
			\{i,j,k\}\in\cT_n \}|=m}}\{x\in \cA_n\,:\, x_{i}=1 \Leftrightarrow i\in S\}, $$
	and write
	\begin{align}\label{pr}
		Z_{n;\a,h}=
		\sum_{m=0}^{\binom{n}{3}} e^{\a\frac{m}{n}} \: \sum_{\ell=0}^{\binom{n}{2}}  \: G^{(n)}_{m,\ell} e^{h\ell},
	\end{align}
where 
 $G^{(n)}_{m,\ell}:=\left|\{S\subseteq \cE_n: |S|=\ell, \\ |\{\{i,j,k\}\subset S\,:\, \{i,j,k\}\in\cT_n \}|=m\}\right|.$ 
Setting $z:=e^{\a}$ and $K^{(n)}_{m,h}:=\sum_{\ell=0}^{\binom{n}{2}}  \: G^{(n)}_{m,\ell} e^{h\ell}$, we obtain 	
\begin{equation}	
		Z_{n;\a,h}(z) =\sum_{m=0}^{\binom{n}{3}}K^{(n)}_{m,h}z^{\frac{m}{n}}\,,
	\end{equation}
which is not a polynomial since $\frac mn$ is not necessarily an integer. For example, when $n=3$ we have $G^{(3)}_{0,0}=G^{(3)}_{1,3}=1$ and $G^{(3)}_{0,1}=G^{(3)}_{0,2}=3$, yielding $Z_{3;\a,h}= (1+3 e^{h}+3e^{2h})+ z^{\nicefrac{1}{3}}e^{3h}$.\\
Instead, by taking the integer part, we obtain the following polynomial representation:
\begin{equation}\label{polinomio}
\hat{Z}_{n;\a,h}\equiv \hat{Z}_{\bar{n}}(z) :=\sum_{k=0}^{\bar{n}}\widetilde{K}^{(n)}_{k,h}z^{k}\,,
\end{equation}
where $\bar{n}:=\lfloor \frac{(n-1)(n-2)}{6}\rfloor$, and  $\widetilde{K}^{(n)}_{k,h}:=\sum_{m: \lfloor \frac{m}{n}\rfloor = k} {K}^{(n)}_{m,h}$.
Note that \eqref{polinomio} can be equivalently written as
\[
\hat{Z}_{\bar{n}}(z) =\widetilde{K}^{(n)}_{\bar{n},h} \sum_{k=0}^{\bar{n}} \frac{\widetilde{K}^{(n)}_{k,h}}{\widetilde{K}^{(n)}_{\bar{n},h}} z^{k}.
\]
Let $\hat{Z}'_{\bar{n}}(z) := \sum_{k=0}^{\bar{n}} \frac{\widetilde{K}^{(n)}_{k,h}}{\widetilde{K}^{(n)}_{\bar{n},h}} z^{k}$.
If $z_{1},\dots,z_{\bar{n}}$ are the complex roots of the polynomial $\hat{Z}'_{\bar{n}}(z)$, then we can write
\[
\hat{Z}'_{\bar{n}}(z) = \prod_{j=1}^{\bar{n}}(z-z_{j}) 
= \prod_{j=1}^{\bar{n}} z_{j} \cdot \prod_{j=1}^{\bar{n}} \left(\frac{z}{z_{j}}-1\right) 
\]
and, since $\prod_{j=1}^{\bar{n}}z_{j}=(-1)^{\bar{n}}\frac{\widetilde{K}^{(n)}_{0,h}}{\widetilde{K}^{(n)}_{\bar{n},h}}$, we get
\[
\hat{Z}'_{\bar{n}}(z) = (-1)^{\bar{n}} (-1)^{\bar{n}} \frac{\widetilde{K}^{(n)}_{0,h}}{\widetilde{K}^{(n)}_{\bar{n},h}} \prod_{j=1}^{\bar{n}}\left(1-\frac{z}{z_{j}} \right) = \frac{\widetilde{K}^{(n)}_{0,h}}{\widetilde{K}^{(n)}_{\bar{n},h}} \prod_{j=1}^{\bar{n}}\left(1-\frac{z}{z_{j}} \right).
\]
Therefore we obtain
\begin{equation}\label{pf_bis}
\hat{Z}_{\bar{n}}(z)= \widetilde{K}^{(n)}_{0,h} \prod_{j=1}^{\bar{n}} \left(1-\frac{z}{z_{j}}\right).
\end{equation}
The following theorem can be now applied to this polynomial representation.  
 \begin{theorem}[\cite{LY}, Thm.~2]\label{teo_LY}
Let $Z_n(z)$ be the polynomial representation of
a partition function.
If there exists a region $R\subseteq\C$
 containing a segment of the real positive axis
that is always root-free then,
 as $ n\to +\infty $ and for $z\in R$,
 all quantities
\begin{equation} \label{derivate}
\frac{1}{n} \ln Z_{n}(z),\quad
\frac{d^k}{d(\ln z)^k}\frac{1}{n}\ln Z_{n}(z),
\,\mbox{ with } k\in\N,
\end{equation}
converge to analytical limits with respect to $z$.
In particular, the limit and derivative operations switch
in the whole region $R$.
\end{theorem} 
\begin{remark}\label{Hp_LY}

Recall from the end of Subsec. \ref{known-res} that $f_{\alpha, h}=\hat{f}_{\alpha,h}=\lim_{n\to \infty} \frac{1}{n^2}\ln \hat{Z}_{\bar{n}}(z)$.
Since the limiting free energy  $f_{\alpha,h}$ is analytic  
for all $(\alpha,h) \in \mathcal{U}_{\alpha,h}^{rs} \setminus \{(\alpha_c,h_c)\}$ (see~\cite{RY}, Thms.~2.1 and 3.9), the partition function \eqref{pf_bis} verifies the hypotheses of Theorem \ref{teo_LY} for all $(\alpha,h)\in \mathcal{U}_{\alpha,h}^{rs}\setminus \{(\a_c,h_c)\}$.
Indeed, at every point where $ f_{\alpha,h}$ is real analytic, it admits a holomorphic extension to a complex neighborhood of that point (see, e.g. \cite[Subsec. 6.4]{KP}). Suppose, by contradiction, that the zeros of the partition function are dense in $\mathbb{R}^+$, meaning that within every arbitrarily small interval on the positive real axis there exists at least one zero. Then, this would contradict the holomorphic extendability of $f_{\alpha,h}$. 
We refer the reader to \cite[Sec. 3]{bena2005} for a comprehensive review on the role played by the location of Yang-Lee zeros on phase transitions.
\end{remark}

\begin{corollary}\label{cor:convergences}
 Let  $(\alpha,h)\in \mathcal{U}_{\alpha,h}^{rs} \setminus \{(\alpha_c,h_c)\}$. Then, 
\begin{equation}\label{lim_secd}
\lim_{n \to +\infty}\frac{6}{n^2}\partial_{\alpha}\hat{f}_{n;\alpha,h}=u^{*3}_{\alpha,h}  \quad \text{ and } \quad \lim_{n \to +\infty}\frac{6}{n^2}\partial_{\alpha\alpha}\hat{f}_{n;\alpha,h} = 3u^{*2}_{\alpha,h}\partial_{\alpha} u^{*}_{\alpha,h}
\end{equation}
\end{corollary}

\begin{proof}
The result is an immediate application of  Thm.~\ref{teo_LY}, which holds true since we are working in the region $\mathcal{U}_{\alpha,h}^{rs} \setminus \{(\alpha_c,h_c)\}$, where the limiting free energy exists and is analytic. We observe that, since in the polynomial representation \eqref{polinomio} we have $z=e^{\alpha}$, then $\frac{d}{d(\ln z)}\frac{1}{n^2}\ln \hat{Z}_{\bar{n}}(z)=\partial_{\alpha}\hat{f}_{n;\a,h}$ and $\frac{d^2}{d(\ln z)^2}\frac{1}{n^2}\ln \hat{Z}_{\bar{n}}(z)=\partial_{\alpha\alpha}\hat{f}_{n;\a,h}$. Therefore, Thm.~\ref{teo_LY}  allows to commute limit and derivative to get
\[
\lim_{n \to +\infty}\frac{6}{n^2}\partial_{\alpha}\hat{f}_{n;\alpha,h} = 6\lim_{n \to +\infty} \partial_\alpha \hat{f}_{n;\alpha,h} = 6\partial_\alpha \left[ \lim_{n \to +\infty}  \hat{f}_{n;\alpha,h} \right] =  6\partial_\alpha f_{\alpha,h} =  u^{*3}_{\alpha,h}.
\]
The second limit on the r.h.s. of \eqref{lim_secd} can be proved in the same way.
\end{proof}

Theorem~\ref{teo_LY} also implies that the derivatives of the finite-size free energy converge locally uniformly.

\begin{proposition}[\cite{BCM}]\label{deriv_convUnif}
Under the hypothesis of Thm. \ref{teo_LY}, the quantities displayed in~\eqref{derivate} converge 
locally uniformly (in $n$) inside the region $R$.
\end{proposition}

\begin{remark}
Recalling \eqref{link_c_avg&var_dens} and the definition \eqref{hat_f} of $\hat{f}_{n;\alpha,h}$, a direct computation shows that 
\[
c'_n(0)=\frac{6}{n^2}\partial_{\alpha}\hat{f}_{n;\alpha,h}  \quad \text{ and } \quad c''_n(0)=\frac{6}{n^2}\partial_{\alpha\alpha}\hat{f}_{n;\alpha,h}.
\]
Therefore, from \eqref{lim_secd}, 
\begin{equation}\label{derivc}
\lim_{n\to \infty}c'_n(0)= u^{*3}_{\alpha,h}  \quad \text{ and } \quad \lim_{n\to \infty}c''_n(0)=3u^{*2}_{\alpha,h}\partial_{\alpha} u^{*}_{\alpha,h}= v(\alpha,h).
\end{equation}
\end{remark}

The proof of our main theorem is then just one step further. We will rely on the analyticity of the free energy and on the uniform convergence of the sequence $(c''_{n}(t))_{n\geq 1}$ guaranteed by Thm. \ref{teo_LY} and Prop. \ref{deriv_convUnif}.
\begin{proof}[Proof of Thm.~\ref{thm_CLT}]
Recall $v(\alpha,h) = 3u^{*2}_{\alpha,h}\partial_{\alpha} u^{*}_{\alpha,h}$ and $W_n=\sqrt{6} \, \frac{\lfloor\nicefrac{ T_n}{n}\rfloor -\hat{\mathbb{E}}_{n;\alpha,h}(\lfloor\nicefrac{T_n}{n}\rfloor)}{n}$. We want to show that 
	\begin{equation}\label{conv:mgf}
		\lim_{n\to+\infty}\hat{\E}_{n;\a,h}(\exp(tW_{n})) =\exp\left(\tfrac{1}{2}v(\alpha,h)t^{2}\right)
	\end{equation}
for all $t\in[0,\eta)$ and some $\eta>0$. 
We aim to  express $\hat{\E}_{n;\a,h}\left(\exp(tW_{n})\right)$ in terms of  $c''_n(t)$. Consider $t> 0$ and set $t_{n}:=\sqrt{6}t/n$. We get
	\begin{align}\label{clt_decomp}
		\ln\hat{\E}_{n;\a,h}(\exp(tW_{n}))
			&=\ln\hat{\E}_{n;\a,h}\left(\exp(t_{n}\trn)\exp\left(-t_n
\hat{\mathbb{E}}_{n;\alpha,h}(\trn)\right)\right)\\
			&\overset{\eqref{cgf},\eqref{link_c_avg&var_dens}}{=} \frac{n^{2}}{6}[c_{n}(t_{n}) - t_{n}c'_{n}(0)]. \notag
	\end{align}
	Notice that, since $c_n(0)=0$, the term in square brackets is the difference between the function $c_n(t_n)$ and its first order Taylor expansion at zero. Therefore, by using Taylor's theorem with Lagrange remainder, one gets
	\[
	\ln\hat{\E}_{n;\a,h}(\exp(tW_{n})) = \frac{c''_{n}(t^{*}_{n}) t^{2}}{2},
	\]
	for some $t^{*}_{n}\in [0,\sqrt{6}t/n]$. To conclude the proof of the central limit theorem, we need to control the limiting behavior of $c''_{n}(t^{*}_{n})$. To this end, we recall from~\eqref{derivc} that $\lim_{n \to \infty} c_n''(0) = v(\alpha, h)$, and that, by Prop.~\ref{deriv_convUnif}, the derivatives of $c_n(t)$ converge locally uniformly. These two properties together yield the following result, which was first proved in a slightly different setting but applies unchanged in the present context.

\begin{lemma}[\cite{BCM}]\label{converg_dsec}
For $(\alpha,h)\in {\mathcal{U}}_{\alpha,h}^{rs}\setminus \{(\alpha_c,h_c)\}$, there exists some $\eta>0$ such that we have $\lim_{n\to+\infty}c''_{n}(t_{n})=v(\alpha,h)$ for all $t_{n}\in[0,\eta)$ with $\lim_{n\to+\infty}t_{n} = 0$.	
\end{lemma}
From the lemma above, we obtain the  convergence of $c''_{n}(t^*_n)$, and, in turn, the convergence of the moment generating function. Therefore  $W_n \overset{d}{\to} \mathcal{N}(0,v(\alpha,h))$ (see \cite{Bill}, Sect.~30).
Finally, the convergence in distribution of 
\begin{equation}
\sqrt{6} \, \frac{\nicefrac{T_n}{n} -\hat{\mathbb{E}}_{n;\alpha,h}(\nicefrac{T_n}{n})}{n}= W_n +  \sqrt{6} \, \frac{\{\nicefrac{T_n}{n}\} -\hat{\mathbb{E}}_{n;\alpha,h}(\{\nicefrac{T_n}{n}\})}{n}
\end{equation}
follows from Slutsky's theorem, as $\sqrt{6} \, \frac{\{\nicefrac{T_n}{n}\} -\hat{\mathbb{E}}_{n;\alpha,h}(\{\nicefrac{T_n}{n}\})}{n}\to 0$ in probability, being the numerator bounded almost surely.
\end{proof}

We conclude the section by shedding a light on an open point of our analysis, i.e. the comparison between \( \hat{\mathbb{P}}_{\alpha,h} \) and \( \mathbb{P}_{\alpha,h} \), the latter being the measure of the edge-triangle model associated with the Hamiltonian~\eqref{Hamilt_ERG}. 

\begin{remark}
To extend the proof of Thm.~\ref{thm_CLT} to the measure $\mathbb{P}_{n;\alpha,h}$ it remains to establish \newline $\lim_{n\to+\infty}\hat{\E}_{n;\a,h}(\exp(tW_{n}))= \lim_{n\to+\infty}\E_{n;\a,h}(\exp(tW_{n}))$, the latter expectation being associated with $\mathbb{P}_{n;\alpha,h}$. A natural approach is to compare the two expectations directly. However, this requires a delicate control over the limiting behavior of the fractional part of the normalized triangle count, which remains an open point of our analysis. 
In particular, with a direct computation one can show that $\mathbb{E}_{n;\alpha,h}(e^{t W_n})= \frac{\hat{\mathbb{E}}_{n;\alpha,h}(e^{tW_n+\alpha\{\tr\}})}{\hat{\mathbb{E}}_{n;\alpha,h}(e^{\alpha\{\tr\}})}$. 
Therefore  
\begin{align*}
|\mathbb{E}_{n;\alpha,h}\left(e^{t W_n}\right)-\hat{\mathbb{E}}_{n;\alpha,h}\left(e^{t W_n}\right)|
&= \frac{\left|\hat{\mathbb{E}}_{n;\alpha,h}\left[e^{t W_n}\left(e^{\alpha\{\tr\}}-\hat{\mathbb{E}}_{n;\alpha,h}(e^{\alpha\{\tr\}})\right) \right] \right|}{\hat{\mathbb{E}}_{n;\alpha,h}(e^{\alpha\{\tr\}})}\\
&\leq \frac{\sqrt{\hat{\mathbb{E}}_{n;\alpha,h}(e^{2t W_n})}\sqrt{\hat{\mathbb{E}}_{n;\alpha,h}(e^{2\alpha\{\tr\}})- \left(\hat{\mathbb{E}}_{n;\alpha,h}(e^{\alpha\{\tr\}})\right)^2}}{\hat{\mathbb{E}}_{n;\alpha,h}(e^{\alpha\{\tr\}})},
\end{align*}
where the last step follows from the Cauchy–Schwarz inequality, by noticing that\newline $\hat{\mathbb{E}}_{n;\alpha,h}\left(e^{\alpha\{\tr\}}-\hat{\mathbb{E}}_{n;\alpha,h}(e^{\alpha\{\tr\}})\right)^2$ is the variance of $e^{\alpha\{\tr\}}$.
This last expression makes it evident that a sufficient condition for the asymptotic equivalence of the two expectations is the convergence $\{\tr\}\to 0$ in probability.
In that case, the dominated convergence theorem would yield the desired conclusion.
\end{remark}

{\bf Acknowledgements.} This research has been partially
funded by the INdAM-GNAMPA projects “Redistribution models on networks” and “Ferromagnetism versus
synchronization: how does disorder destroy universality?”.

\bibliographystyle{abbrv}
\bibliography{biblio}
\end{document}